\documentclass[12pt]{article}
\usepackage{amsmath}
\usepackage{amsfonts}
\usepackage{amssymb}
\usepackage{amsthm}
\usepackage{dsfont}
\newtheorem{thm}{Theorem}

\newtheorem{lem}{Lemma}
\title{Are fractional Brownian motions predictable?}

\author{Adam Jakubowski
\\Nicolaus Copernicus University\\
adjakubo@mat.uni.torun.pl}

\begin{document}
\maketitle

\begin{abstract}
We provide a device, called the local predictor, which extends the idea
of the predictable compensator. It is shown that a fBm with the Hurst index
greater than $1/2$ coincides
with its local predictor while fBm with the Hurst index smaller than $1/2$
does not admit any local predictor.
\end{abstract}

\section{Intoduction}

The question in the title is provocative, of course. Everybody
familiar with the theory of stochastic processes knows that a
continuous adapted process on the stochastic basis
$(\Omega,\mathcal{F}, \{\mathcal{F}_t\}, P)$ is predictable, in the
sense it is measurable with respect to the $\sigma$-algebra of {\em
predictable subsets} of $\Omega\times\mathds{R}^+$. And fractional
Brownian motions {\em are continuous}.

The point is that the predictability has a  clear meaning in the
discrete time, while in continuous time it looses its intuitive
character. Brownian motion serves in many models  as a source of
unpredictable behavior, but it is predictable in the sense of the
general theory of processes.

We are not going to suggest any change in the established
terminology, although the old alternative of ``well-measurable"
sounds more reasonable. Our aim is to provide a device for verifying
whether some fractional Brownian Motions are ``more predictable"
than others.

\section{The local predictor and its existence for fBms}

We develop the idea of a predictable compensator in somewhat
unusual direction. Let, as before,  $(\Omega,{\mathcal
F},\{{\mathcal F}_t\}_{t\in[0,T]},P)$ be a stochastic basis,
satisfying the ``usual" conditions, i.e. the filtration $\{{\mathcal
F}_t\}$ is right-continuous and ${\mathcal F}_0$ contains all
$P$-null sets of ${\mathcal F}_T$. By convention, we set ${\mathcal
F}_{\infty} = {\mathcal F}$.

Let $\{X_t\}_{t\in [0,T]}$ be a stochastic process on
$(\Omega,{\mathcal F},P)$, adapted to $\{{\mathcal
F}_t\}_{t\in[0,T]}$ (i.e. for each $t\in [0,T]$, $X_t$ is ${\mathcal
F}_t$ measurable) and with c\`adl\`ag (or regular) trajectories
(i.e. its $P$-almost all trajectories are right-continuous and
possess limits from the left on $(0,T]$).

Suppose we are sampling the process $\{X_t\}$ at points $0=t_0^{\theta} < t_1^{\theta}
< t_2^{\theta} < \ldots < t_{k^{\theta}}^{\theta} = T\}$ of a partition $\theta$ of the
interval $[0,T]$. By the discretization of $X$ on $\theta$ we
mean the process

\[ X^{\theta}(t) = X_{t_k^{\theta}}\textrm{\quad if\quad }
t_k^{\theta} \leq t < t_{k+1}^{\theta},\
X^{\theta}_T = X_T.\]

If random variables $\{X_t\}_{t\in[0,T]}$ are integrable,  we can
associate with any discretization $X^{\theta}$ its ``predictable
compensator"
\begin{eqnarray*}
 A^{\theta}_t &=& 0\textrm{\quad if\quad } 0 \leq t < t^{\theta}_1,\\
 A^{\theta}_t &=& \sum_{j=1}^k E\big(X_{t^{\theta}_j} - X_{t^{\theta}_{j-1}}
 \big|{\mathcal F}_{t^{\theta}_{j-1}}\big)
\textrm{\quad if\quad }t^{\theta}_k \leq t < t^{\theta}_{k+1},\
k=1,2, \ldots,
k^{\theta} - 1,\\
A^{\theta}_T &=& \sum_{j=1}^{k^{\theta}} E\big(X_{t^{\theta}_j} -
X_{t^{\theta}_{j-1}}\big|{\mathcal F}_{t^{\theta}_{j-1}}\big)\,.
\end{eqnarray*}
Notice that $A^{\theta}_t$ is ${\mathcal
F}_{t^{\theta}_{k-1}}$-measurable for $t^{\theta}_k \leq t <
t^{\theta}_{k+1}$, and so the processes $A^{\theta}$ are predictable
in a very intuitive manner, both in the discrete and in the
continuous case. It is also clear, that the discrete-time process
$\{M^{\theta}_t\}_{t\in \theta}$ given by
\[ M^{\theta}_t = X^{\theta}_t - A^{\theta}_t,\quad t\in\theta,\]
is a martingale with respect to the discrete filtration $\{{\mathcal
F}_t\}_{t\in \theta}$.

If we have square integrability of $\{X_t\}_{t\in [0,T]}$, then the predictable
compensator $\{A^{\theta}_t\}_{t\in\theta}$ possesses also a clear variational interpretation.
Fix $\theta$ and let $\mathcal{A}^{\theta}$ be the set of discrete-time stochastic processes
$\{A_t\}_{t\in\theta}$ which are $\{\mathcal{F}_t\}_{t\in\theta}$-predictable, i.e. for each
$t = t^{\theta}_k\in\theta$, $A_{t^{\theta}_k}$ is $\mathcal{F}_{t^{\theta}_{k-1}}$-measurable.
Then the predictable compensator $\{A^{\theta}_t\}_{t\in\theta}$ minimizes the functional
\[ \mathcal{A}^{\theta} \ni A \mapsto E [ X - A]_T,\]
where the discrete quadratic variation $[\cdot]$ is defined as usual by
\[ [Y]_T = \sum_{t\in\theta} (\Delta Y_t)^2 = \sum_{k=1}^{k^{\theta}} (Y_{t^{\theta}_k} -
Y_{t^{\theta}_{k-1}})^2.\]

Now consider a sequence $\Theta =\{\theta_n\}$ of normally
condensing partitions of $[0,T]$. This means we assume  $\theta_n
\subset \theta_{n+1}$ and the mesh
\[|\theta_n| = \max_{1\leq k \leq k^{\theta_n}} t^{\theta_n}_k - t^{\theta_n}_{k-1} \to 0,\quad
\textrm{as\ } n\to\infty.\]
We will say that an adapted stochastic
process $\{X_t\}_{t\in[0,T]}$ with regular trajectories admits {\em
a local predictor} $\{C_t\}_{t\in[0,T]}$  {\em along} $\Theta =
\{\theta_n\}$ and in the sense of convergence $\to_{\tau}$ if
\[ A^{\theta_n} \to_{\tau} C\]
and $C$ has regular trajectories.

As an example we will examine the existence of a local predictor for fractional
Brownian motions.

Let us recall that a fractional Brownian motion (fBm) $\{B^H_t\}_{t\in\mathds{R}^+}$ of Hurst
index $H\in(0,1)$ is a continuous and centered Gaussian process with covariance function
\[ E ( B^H_t B^H_s) = \frac{1}{2} (t^{2H} + s^{2H} - |t - s|^{2H}).\]
For extensive theory of stochastic analysis based on fBms we refer to the most recent monographs
\cite{BHOZ} and \cite{Mish}.
\begin{thm}\label{ajth1}
For $H \in (1/2,1)$ the fractional Brownian motion $\{B^H_t\}_{t\in[0,T]}$ coincides
with its local predictor along any
sequence of normally condensing partitions and
in the sense of {\em the uniform convergence in probability}.
\end{thm}

\begin{proof}
We consider the natural filtration $\{\mathcal{F}_t\}_{t\in[0,T]}$ generated by the
fBm $\{B^H_t\}$. Let $\{\theta_n\}$ be a sequence of normally condensing
partitions of $[0,T]$ and let $\{A^{\theta_n}_t\}_{t\in\theta_n}$ be the predictable
compensator for the discretization of $\{(B^H)^{\theta_n}_t\}$ on $\theta_n$.
By the Doob inequality
\begin{eqnarray*}
E \sup_{t\in\theta_n} ((B^H)^{\theta_n}_t - A^{\theta_n}_t)^2 &\leq& 4 E (B^H_T - A^{\theta_n}_T)^2
= 4 E [(B^H)^{\theta_n} - A^{\theta_n}]_T \\
&\leq& 4 E [(B^H)^{\theta_n}]_T
= 4 \sum_{k=1}^{k^{\theta_n}} |t^{\theta_n}_k - t^{\theta_n}_{k-1}|^{2H}\\
&\leq&
4 T |\theta_n|^{2H-1} \to 0.
\end{eqnarray*}
Since we have also almost surely
\[ \sup_{t\in [0,T]} | (B^H)^{\theta_n}_t - B^H_t | \to 0,\]
the theorem follows.
\end{proof}
The above result is a direct consequence of the fact that for $H\in
(1/2.1)$ the fBm is a process of {\em energy zero in the sense of
Fukushima} \cite{Fu80}, i.e.
\[ E [X^{\theta_n}]_T = E \sum_{k=1}^{k^{\theta_n}} (X_{t^{\theta_n}_k} -
X_{t^{\theta_n}_{k-1}})^2 \to 0, \quad \textrm{as\ } n\to\infty.\]

Hence we have also

\begin{thm}\label{ajth2}
If $\{X_t\}$ is continuous adapted and of energy zero in the sense
of Fukushima, then it coincides with its local predictor along any
sequence of condensing partitions and in the sense of {\em the
uniform convergence in probability}.
\end{thm}

It may be instructive to write down the assertion of Theorems \ref{ajth1} and \ref{ajth2}.
\begin{equation}\label{ajeq1}
\sup_{t\in [0,T]} | X_t - A^{\theta_n}_t | \to_P 0.
\end{equation}

Jacod in \cite[p. 94]{Jac84}, in the context of so-called processes admitting a tangent
process with independent increments, introduced a class $B(\{\theta_n\})$ of continuous
bounded predictable processes satisfying \eqref{ajeq1} and
\begin{equation}\label{ajeq2}
\sum_{\{k\,:\,t^{\theta_n}_{k+1} \leq\, t\}} E
((X_{t^{\theta_n}_{k+1}} - X_{t^{\theta_n}_{k}})^2|
\mathcal{F}_{t^{\theta_n}_{k}}) - (E (X_{t^{\theta_n}_{k+1}} -
X_{t^{\theta_n}_{k}}| \mathcal{F}_{t^{\theta_n}_{k}}))^2
 \to_P 0.
\end{equation}
The class $B(\{\theta_n\})_{\text{loc}}$, containing fBms for
$H\in(1/2,1)$, was also considered in \cite{Jac84}. But fBms did not
appear in Jacod's paper.

For martingales we have a rather simple statement.
\begin{thm}\label{thmart}
The local predictor of a martingale (in particular: of a Brownian motion)
trivially exists and equals 0.
\end{thm}

It is interesting that for $H \in (0,1/2)$ the compensators of discretizations of  fBms explode.
\begin{thm}\label{thdwa}
For $H \in (0,1/2)$ the fractional Brownian motion $\{B^H_t\}_{t\in[0,T]}$ admits no
local predictor. In fact, for any sequence $\{\theta_n\}$ of normal condensing partitions
we have
\[ \sup_n E (A^{\theta_n}_T)^2 = +\infty.\]
\end{thm}
\begin{proof}
It suffices to show that
\begin{equation}\label{ajeq3}
\sup_n E (B^H_T - A^{\theta_n}_T)^2 = \sup_n E [(B^H)^{\theta_n} - A^{\theta_n}]_T = +\infty.
\end{equation}
For that we need a lemma, which is basically a result of Nuzman and Poor
\cite[Theorem 4.4]{NuPo}, with
corrections due to Anh and Inoue \cite[Theorem 1]{AnIn}.
\begin{lem}\label{ajL1}
If $H\in(0,1/2)$ then for $0\leq s < t$ there exists a nonnegative function
$h_{t,s}(u)$ such that
\begin{equation}
\int_0^s h_{t,s}(u)\,du = 1,
\end{equation}
and
\begin{equation}
E (B^H_t|\mathcal{F}_s) = \int_0^s h_{t,s}(u) B^H_u du,\ \text{a.s.}
\end{equation}
Recall we work with  the natural filtration $\mathcal{F}_s =
\sigma\{B^H_u\, : \, 0\leq u \leq s\}$.
Note also that it is possible to write down the exact (and complicated) form of the
function $h_{t,s}$, but we do not need it.
\end{lem}

We need also a remarkably simple lower bound for conditional variances.
\begin{lem}\label{ajL2}
For $H\in(0,1/2)$ and $0\leq s < t$
\begin{equation}\label{ajeq4}
E (B^H_t - E(B^H_t|\mathcal{F}_s))^2 = E (B^H_t - B^H_s - E(B^H_t - B^H_s|\mathcal{F}_s))^2
\geq \frac12 |t - s|^{2H}.
\end{equation}
\end{lem}
\begin{proof} Inequality (\ref{ajeq4}) follows from the chain of equalities
\begin{eqnarray*}
\lefteqn{E (B^H_t - B^H_s - E(B^H_t - B^H_s|\mathcal{F}_s))^2 =
E (B^H_t - B^H_s)^2 - E(E(B^H_t - B^H_s|\mathcal{F}_s))^2}\\
&=& E (B^H_t - B^H_s)^2 - E((B^H_t - B^H_s)E(B^H_t - B^H_s|\mathcal{F}_s))\\
&=& E (B^H_t - B^H_s)^2 - E( B^H_t E(B^H_t|\mathcal{F}_s)) - E (B^H_s)^2\\
& & \qquad + E (B^H_s E(B^H_t|\mathcal{F}_s))
+ E B^H_t B^H_s \\
&=& (t-s)^{2H} - \frac12 \int_0^s h_{t,s}(u) (t^{2H} + u^{2H} - (t-u)^{2H})\,du - s^{2H}\\
& &+ \frac12 \int_0^s h_{t,s}(u) (s^{2H} + u^{2H} - (s-u)^{2H})\,du
+\frac12 (t^{2H} + s^{2H} - (t-s)^{2H})\\
&=& \frac12 (t-s)^{2H} + \frac12\int_0^s h_{t,s}(u) ((t-u)^{2H} - (s-u)^{2H})\,du,
\end{eqnarray*}
and from the observation that for $H \in (0,1/2)$
\[\frac12\int_0^s h_{t,s}(u) ((t-u)^{2H} - (s-u)^{2H})\,du \geq 0.\]
\end{proof}
Now we are ready to verify \eqref{ajeq3}. By \eqref{ajeq4}
\[E [(B^H)^{\theta_n} - A^{\theta_n}]_T \geq \frac12 \sum_{k=1}^{k^{\theta_n}}
|t^{\theta_n}_k - t^{\theta_n}_{k-1}|^{2H} \to +\infty,\]
for every sequence $\{\theta_n\}$ of normal condensing partitions of $[0,T]$.
\end{proof}

\noindent {\em Remark 2.7} The random variables $A^{\theta_n}_T$ are
Gaussian, so $\sup_n E (A^{\theta_n}_T)^2 = +\infty$ is equivalent
to the lack of tightness of the family $\{A^{\theta_n}_T\}$. Thus in
the case $H\in(0,1/2)$ the compensators do not stabilize in any
reasonable probabilistic sense.

\section{On the existence of local predictors}

\subsection{Submartingales}
It is not difficult to show that any continuous and nondecreasing adapted integrable
process coincides with its local predictor in the sense of the uniform convergence
in probability. This implies in turn that any submartingale of class D with continuous
increasing process in the Doob-Meyer decomposition also admits a local predictor
which coincides with its predictable continuous compensator.

This is no longer true if the compensator is discontinuous. We have then in general only
weak in $L^1$ convergence of discrete compensators. Such convergence, although satisfactory
from the analytical point of view, brings only little probabilistic understanding
to the nature of the compensation.

To overcome this difficulty, the author proposed in \cite{Jak05} an approach based on
the celebrated Koml\'os theorem \cite{Koml}. It is proved \textit{ibidem} that given
any sequence $\{\theta_n\}$ of partitions one can find a subsequence $\{n_j\}$ along which
the C\'esaro means of compensators of discretizations converge to the limiting compensator.
More precisely, if $\{n_j\}$ is the selected subsequence and we denote by $\{A^j_t\}$ the
predictable compensator of the discretization on $\theta_{n_j}$, then for each rational
$t\in [0,T]$
\begin{equation}\label{ajeqxx}
 B^N_t = \frac1N \sum_{j=1}^N A^j_t \to A_t,\quad \text{a.s.},
 \end{equation}
where $A$ is the continuous-time process in the Doob-Meyer decomposition.
In fact the above convergence can be strengthened: for each stopping time $\tau \leq T$
we have
\begin{equation}\label{ajeqx}
\limsup_{N \to +\infty} B^N_{\tau} = A_{\tau},\quad \text{a.s.}.
\end{equation}
In particular, this directly implies predictability of $\{A_t\}$.

\subsection{Processes with finite energy and weak Dirichlet processes}

Graversen and Rao \cite{GrRa85} proved the Doob-Meyer type decomposition for
a wide class of \textit{processes with finite energy}.
Examples of how such decomposition can work in the framework of
\textit{weak Dirichlet processes} (including cases of uniqueness) were provided
in several recent papers (see \cite{CJMS06}, \cite{ER1}, \cite{ER2}, \cite{GoRu06}).
Similarly as in the general theory for submartingales, in the Graversen-Rao original
paper the existence of the predictable
decomposition was obtained by the weak-$L^2$ arguments.

The author proved in \cite{Jak06} that the Koml\'os machinery works perfectly also in this
problem. For a sequence $\{\theta_n\}$ of partitions of $[0,T]$ such that random variables
 $\{A^{\theta_n}_T\}$ are \textit{uniformly integrable} one can select a subsequence such that
for each stopping time $\tau \leq T$
\[ B^N_{\tau} \to A_{\tau},\quad \text{in $L^1$}.\]
In the above we use the setting of \eqref{ajeqxx} and \eqref{ajeqx}.

In \cite{Jak06} an example of a bounded process was given, for which the terminal values
$\{A^{\theta_n}_T\}$ were not uniformly integrable. It follows from our Theorem \ref{thdwa}
that the fractional Brownian motion with the Hurst index $H\in(0,1/2)$ is another, more
natural example of such phenomenon.


\subsection*{Acknowledgment}
The author is grateful to Esko Valkeila for stimulating discussions.
\end{document}